\documentclass{amsart}
\usepackage{amssymb}
\usepackage{amsfonts}
\usepackage{graphicx}
\newtheorem{thm}{Theorem}[section]
\newtheorem{lem}[thm]{Lemma}
\theoremstyle{plain}
\newtheorem{cor}[thm]{Corollary}
\newtheorem{defi}[thm]{Definition}

\newtheorem{remk}[thm]{Remark}
\newtheorem{exa}[thm]{Example}

\begin{document}
		\title {On the $m$-graph of a finite Abelian Group}
		\author{ Ayman Badawi}
	\address{department of Mathematics  $\&$ Statistics, The American University of Sharjah, P.O.  Box 26666, Sharjah, United Arab Emirates.}
	\email{abadawi@aus.edu}
	\date{\today}

\begin{abstract} Let $H$ be a finite abelian (commutative) group of order $n \geq 2$, and $m >1$ be an integer. We define the $m$-graph of $H$, denoted by $m-G(H)$, as a simple undirected graph with vertex set $H$, and two distinct vertices, $a, b \in H$, are connected by an edge if and only if $a^m = b$ or $b^m = a$. Several results regarding the properties of the $m$-$G(H)$ have been established. 
\end{abstract}
\maketitle
\section{Introduction}
 In the last thirty years or so, there has been considerable attention to graphs from rings; for example, see \cite{B1}, which gives an overview of three decades of research in the subject.  Graphs from groups are also well-studied, and many research papers are devoted to them.  For example, see \cite{AS} and \cite{T1}.  This paper introduces a new graph related to the power graph as in \cite{K1} and \cite{C1}, but with a completely different structure and properties. Let $G$ be a finite group. Then the {\it directed power graph} of $G$ as in \cite{K1}, is a digraph with vertex set $G$ and a vertex $a\in G$ is adjacent to a vertex $b\in G$ if $b = a^r$ for some integer $ r\geq 1$. This concept was extended in \cite{C1} to an undirected power graph. We recall from \cite{C1} that the undirected power graph of $G$ or simply the {\it power graph} of $G$, denoted by $P(G)$, is a simple undirected graph with vertex set $G$ and two distinct vertices $a, b \in G$ are adjacent (i.e. connected by an edge) if $a = b^r$ or $b = a^r$. Power graphs of groups studied extensively by many authors, for example, see \cite{A1}, \cite{C0}, \cite{C1}, \cite{K1}, \cite{K2}, \cite{P1}, \cite{P2}, \cite{P3}, \cite{P4}, and \cite{T1}. Excellent survey articles with extensive lists of references on power graphs are \cite{A1} and \cite{C0}. Let $G$ be a finite abelian group of order $n \geq 2$. Then it is easy to see that the $P(G)$ is connected and the diameter of $P(G) \leq 2$. The authors in \cite[Theorem 4]{T1} proved that the $P(G)$ is a tree if and only if $P(G) = K_{1, n-1}$ if and only if every element in $G$ is its inverse.  The graph of $G$ that we will introduce in this paper is not always connected; if it is, it is a tree, and its diameter is an integer $1 \leq d < \infty$.  
 
 In this paper, we introduce and study the $m$-graph of a finite abelian (commutative) group. Let $H$ be a finite abelian group of order $n \geq 2$, $m > 1$, and $k = gcd(m, n) $.  We define the {\it $m$-$G(H)$ } to be a simple undirected graph with vertex set $H$ and two distinct vertices, a, b, are adjacent (i.e., connected by an edge) if $b = a^m$ or $a = b^m$.  We show (Theorem \ref{T1}) that the $m$-$G(H)$ is connected if and only if every prime factor of $n$ is a prime factor of $m$, and hence every prime factor of $n$ is a prime factor of $k$. We show (Theorem \ref{T3} and Theorem \ref{T5}(7)) if the $m$-$G(H)$ is connected, then it is a tree.  We show (Theorem \ref{Tciso} and Theorem \ref{TNCiso}) if the $m$-$G(H)$ is connected, then the $m$-$G(H)$ is graph-isomorphic to the $k$-$G(H)$. If the $m$-$G(H)$ is connected, we compute (Theorem \ref{T2} and Theorem \ref{T5}) the degrees of its vertices, and we compute (Theorem \ref{CDIM} and Theorem \ref{NCDIM}) its diameter. Many examples are provided. We show (Theorem \ref{Tree1} and Theorem \ref{Tree2}) that certain trees can be represented as the $m$-$G(H)$ of a cyclic group $H$ for some $m \geq 2$. This paper is devoted to studying graph properties and group structures of the connected $m$-$G(H)$.
 
 We recall some definitions from graph theory.  Let $G$ be a simple undireted graph with vertex set $V$. Then the {\it distance} between two distinct vertices $a, b$ $\in V$, denoted by $d(a, b)$, is the length of the shortest path between $a$ and $b$; if there is no such path, then we define $d(a, b) = \infty$. The {\it diameter} of $G$, abbreviated as $diam$($G$), is max$\{d(a, b) \mid a, b \in V\}$. The graph $G$ is called {\it connected} if there is a path between every two distinct vertices of $G$, and it is {\it disconnected} if there are two distinct vertices of $G$ that are not connected by a path in $G$. The graph $G$ is called a {\it tree} if it is connected and has no cycles. The graph $G$ is called a {\it complete bipartite} graph, denoted by $K_{n, m}$, If its set of vertices can be split into two sets $A$ and $B$, such that every two vertices in $A$ and every two vertices in $B$ are not adjacent, still, every vertex in $A$ is adjacent to every vertex in $B$, $|A| = n$, and $|B| = m$.  The {\it chromatic} number of $G$  is the smallest number of colours needed to colour the vertices of $G$ so that no two adjacent vertices receive the same colour. Let $v$ be a vertex of $G$. Then the {\it degree} of $v$, abbreviates as $deg(v)$, is the number of the vertices that are adjacent to $v$. 
 \begin{remk}
 Assume that $G$ is a tree. It is known that between any two distinct vertices of $G$, there is exactly one path. Hence, $diam$($G$) is the length of the longest path of $G$. 
 \end{remk}
 
\section{General results and examples}
\begin{exa}\label{E1}
The following are graphs of the $2-G(Z_4)$ and the $2-G(Z_6)$. Note that the first graph is connected while the second is disconnected.

\includegraphics[height = 5cm, width = 12cm]{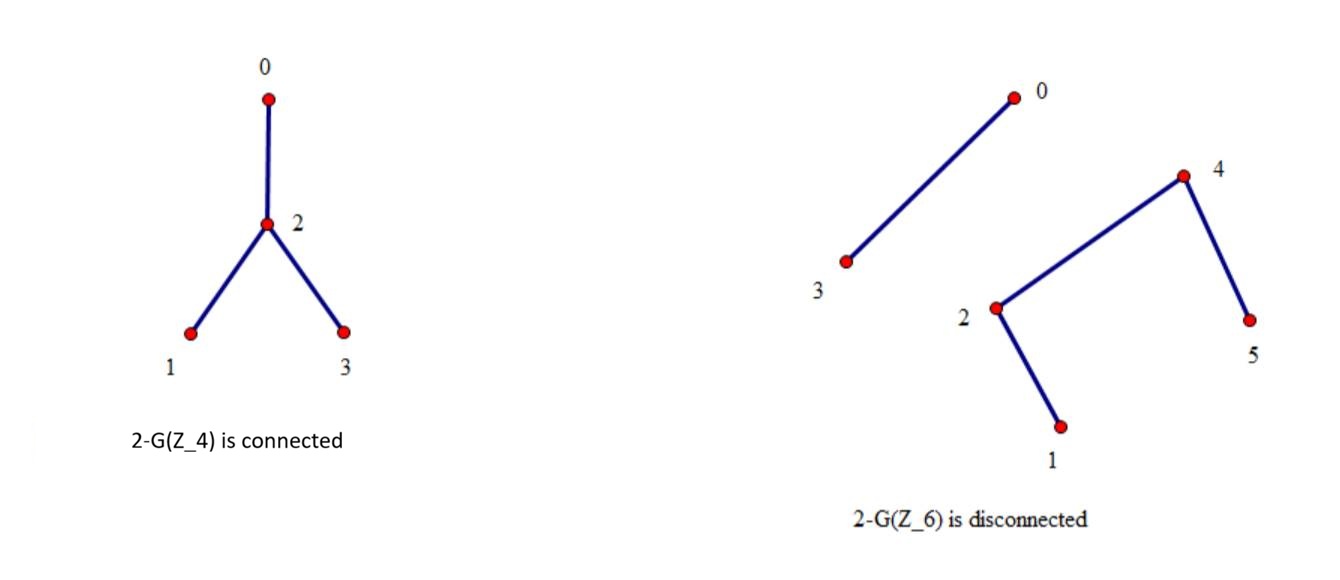}

\end{exa}
In view of example \ref{E1}, we have the following result.
\begin{thm}\label{T1}
	Let $H$ be a finite abelian group of order $n \geq 2$, $m > 1$ be an integer, and $k = gcd(m, n)$. The following statements are equivalent.
	\begin{enumerate}
		\item $m-G(H)$ is connected.
		\item $k\mid n$ and if $p$ is a prime factor of $n$, then $p \mid k$ (and hence $p \mid m$).
	\end{enumerate}
\end{thm}
\begin{proof} $(1)\Rightarrow (2)$. Since $k = gcd(m, n)$, it is clear that $k \mid n$. Assume the $m-G(H)$ is connected. Let $p$ be a prime factor of $n$ such that $p\nmid k$. Then we know that $H$ has an element, say $a$, such that the order of $a$ is $p$. We show there is no path from $a$ to $0$ (the identity of $H$). Assume there is a path $P$ from $a$ to $0$. Then
	$$P: a--a^m--a^{m^2}-- \cdots -- a^{m^w} = 0$$ for some positive integer $w \geq 1$. Thus $p \mid m^w$, a contradiction, since $p \nmid m$. Thus, every prime factor of $n$ is a prime factor of $k$ (and hence a prime factor of $m$).
	
	$(2)\Rightarrow (1)$. Assume that $k\mid n$ and if $p$ is a prime factor of $n$, then $p \mid k$. We show that $m-G(H)$ is connected. It suffices to show that every vertex $a \in H$ is connected by a path to the identity $0$ of $H$. Let $a \in H$ such that $a \not = 0$. Let $d \geq 2$ be the order of $a$ in $H$, i.e., $d$ is the smallest positive integer such that $a^d = 0$ in $H$. Since every prime factor of $n$ is a prime factor of $m$, let $w \geq 1$ be the least positive integer such that $m^w = df$ for some integer $f\geq 1$, and hence $a^{m^w} = 0$ in $H$. Thus $P : a--a^m--a^{m^2} \cdots -- a^{m^w} = 0$ is a path from $a$ to $0$. Hence, the $m-G(H)$ is connected.
	
\end{proof}

\begin{remk} Since $3 \mid 6$ and $3\nmid 2$, by Theorem \ref{T1}, the $2-G(Z_6)$ in example \ref{E1} is disconnected.
	
	\end{remk}

In light of Theorem \ref{T1}, we have the following result.
\begin{cor}\label{C1}
	Let $H$ be a finite abelian group of order $n$, where $n >1 $ is a square-free integer. The following statements are equivalent.
	\begin{enumerate}
		\item The $m-G(H)$ is connected for some integer $m > 1$.
		\item $m = fn$ for some integer $f \geq 1$.
	\end{enumerate}
\end{cor}
\begin{proof}
$(1)\Rightarrow (2)$. Assume the $m-G(H)$ is connected for some integer $m > 1$. Then every prime factor of $n$ is a prime factor of $m$ by Theorem \ref{T1}. Since $n$ is square-free, we conclude that $m = fn$ for some integer $f \geq 1$.

$(2) \Rightarrow (1) $.  Assume that $m = fn$ for some integer $f \geq 1$. Then $k = gcd(m, n) = n$, and every prime factor of $n$ is a factor of $k$. Hence the $m-G(H)$ is connected by Theorem \ref{T1}.

	\end{proof}
The following is the $24-G(Z_6)$.
\begin{exa}\label{E2}
	Let $H = Z_6$ and $m = 24$, note that 6 is a square-free integer and $m = (4)(6) = 4n$. Hence, the $24-G(Z_6)$ is connected by Corollary \ref{C1}. 
	
	\includegraphics[height = 5cm, width = 10cm]{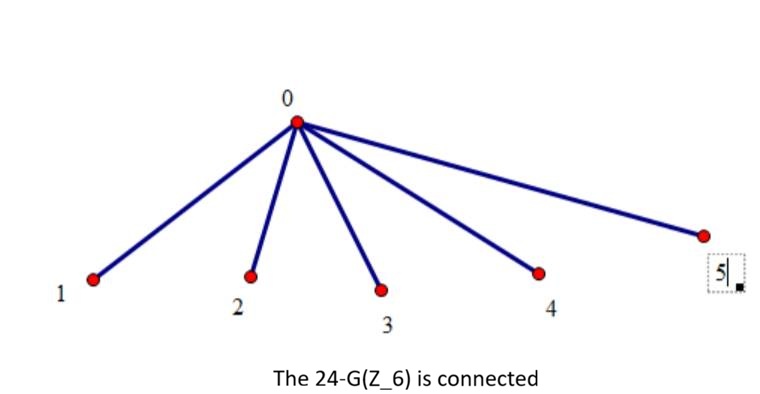}

\end{exa}
The following result is needed.
\begin{lem}\label{L1}
	Let $n, m > 1$ be integers and assume that the equation $mx = a$ has a solution in $Z_n$ for some $a\in Z_n$. Let $k = gcd(m, n) $. Then $k \mid a$ in $Z$ and the equation $mx = a$ in $Z_n$ has exactly $k$ distinct solutions in $Z_n$.  Furthermore, assume that every prime factor of $n$ is a prime factor of $m$ and $a \not = 0$. Then $ma \not = a$ in $Z_n$.	
	\end{lem}
	\begin{proof}
		It is known from a basic course in number theory that $mx \equiv a \ \ (mod \ \ n)$ has a solution if and only if  $k = gcd(m, n) \mid a$ in $Z$, and if it has a solution, then it has exactly $k$ distinct solutions in $Z_n$. Assume that every prime factor of $n$ is a prime factor of $m$, and $ma  = a$ in $Z_n$. Then $n \mid a(m-1)$ in $Z$. Since every prime factor of $n$ is a prime factor of $m$, we conclude that $gcd(m-1, n) = 1$. Thus $n |a $, and hence $a = 0$ in $Z_n$.
		
		\end{proof}

\section{Properties of the $m$-$G(H)$ when $H$ is cyclic}

	We have the following result.
		
	\begin{thm}\label{T2}
		Let $H$ be a finite cyclic group of order $n > 1$, i.e., $H = Z_n$, and assume the $m-G(H)$ is connected for some integer $m > 1$. Let $k = gcd(m, n)$ and $a \in Z_n$. Then
		\begin{enumerate}
			\item $deg(0) = k - 1$.
			\item If $a \not = 0$  and $k \nmid a$ in $Z$, then $deg(a) = 1$. 	
			\item If $a \not = 0$ and $k \mid a$ in $Z$, then $deg(a) = k + 1$.
	\end{enumerate}
\end{thm}
\begin{proof} First, observe that every prime factor of $n$ is a prime divisor of $m$ by Theorem \ref{T1}.
	
	$(1)$. Since $k \mid n$ and $H$ is cyclic, $H$ has a unique subgroup, say $F$, of order $k$. By Lemma \ref{L1}, it is clear that $F$ is the solution set to the equation $mx = 0$ in $Z_n$. Since $0 \in F$ and every $b \in F - \{0\}$ is adjacent to $0$, we conclude that $deg(0) = |F| - 1 = k - 1$.
	
	$(2)$.  Assume that $ k\nmid a$ in $Z$. Then the equation $mx = a$ has no solution in $Z_n$. Let $b = ma \in Z_n$. Note that $b \not = a$ by Lemma \ref{L1} and $b$ is adjacent to $a$. Hence $deg(a) = 1$.
	
	$(3)$.  Assume that $a \not = 0$ and $k \mid a$ in $Z$. Let $F$ be the solution set to the equation $mx = a$ in $Z_n$. Then $|F| = k$ and $a \not \in F$ by Lemma \ref{L1}. Let $b = am \ \ (mod \ \ n) \in Z_n$. Then  $ma = b$ in $Z_n$ and $b$ is adjacent to $a$. Since $b \not = a$ by Lemma \ref{L1}. Thus $deg(a) = |F| + 1 = k + 1$.

\end{proof}

The following example shows that $H$ is cyclic in Theorem \ref{T2} is crucial.

\begin{exa}\label{E3}
	Let $H = Z_2 \times Z_4$. Then $H$ is a finite abelian group of order $8$ that is not cyclic. Then the $2-G(H)$ is
	
		\includegraphics[height = 8cm, width = 12cm]{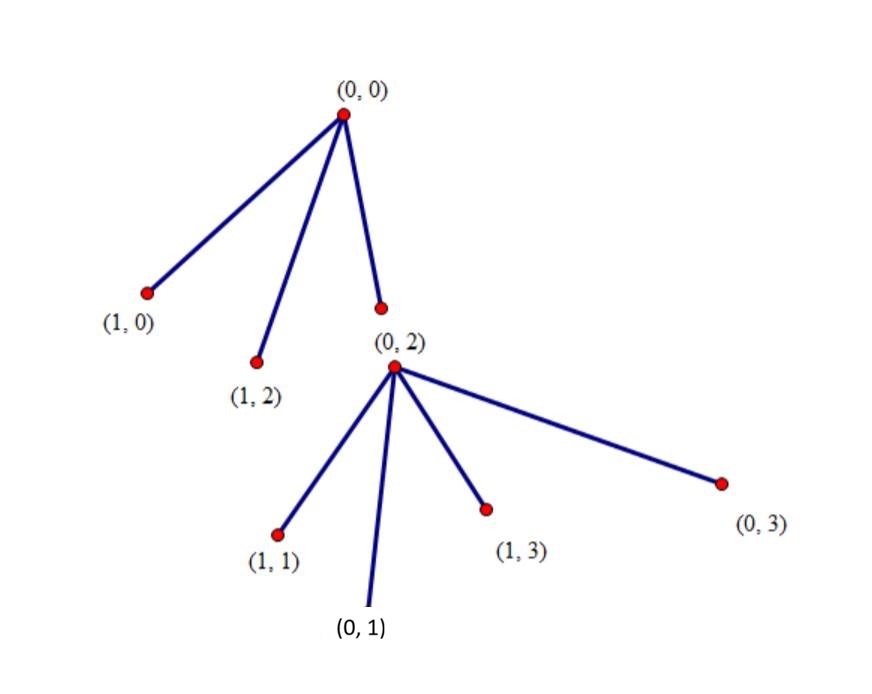}
	
	Note that $m = k = gcd(2, 8) = 2$, $deg(0, 0) = 3 \not = k -1$ , and $deg(0, 2) = 4 \not = k +1 $.
\end{exa}	
	We need the following result.
	
	\begin{lem}\label{L2}
		$m, n \geq 2$ be positive integer and $k = gcd(m, n)$.  Assume that every prime factor of $n$ is a prime factor of $m$. Let $a \in Z_n$ and $i\geq 1$ be a positive integer. Then $k^ia = 0$ in $Z_n$ if and only  $m^ia =0$ in $Z_n$.
	\end{lem}
	\begin{proof}
		Since $k = gcd(m, n)$, we have $n = dk$ and $m = qk$ for some positive integers $d, q\geq 1$ such that $gcd(d, q) = 1$.
		
	Let $a \in Z_n$ and $i\geq 1$ be a positive integer. If $k^ia = 0$ in $Z_n$, then $m^ia = q^ik^ia = 0$ in $Z_n$. Assume that  $m^ia = q^ik^ia = 0$ in $Z_n$. Since $gcd(d, q^i) = 1$, we have $1 = r_1q^i + r_2d$ for some integers $r_1, r_2$. Hence $k^ia = r_1q^ik^ia + r_2dk^ia$. Since $n \mid r_1q^ik^ia$ by hypothesis and $n \mid r_2dk^ia$, we conclude that $n \mid k^ia$. Thus $k^ia = 0$ in $Z_n$.
	\end{proof}
	\begin{thm}\label{T3}
				Let $H$ be a finite cyclic group of order $n > 1$, i.e., $H = Z_n$, and assume the $m-G(H)$ is connected for some integer $m > 1$. Then the $m$-$G(H)$ ($k$-$G(H)$) is a tree, and hence the chromatic number of the $m$-$G (H)$ ($k$-$G(H)$) is $2$.
					\end{thm}
	
	\begin{proof}
		First, observe that every prime factor of $n$ is a prime divisor of $m$ by Theorem \ref{T1}.
		Let $k = gcd(m, n)$. Then $n = qk$ for some integer $q \geq 1$. Let $D = \{v \in Z_n \mid deg(v) = k+1\}$. Then $D = \{k, ... , (q-1)k\}$ by Theorem \ref{T2}(3). Hence $|D| = q - 1$. Let $F = \{v \in Z_n \setminus \{0\} \mid deg(v) = 1\}$. Then $F = \{v \in Z_n \setminus \{0\} \mid k \nmid v\}$ by Theorem \ref{T2}(2). Then $|F| = n - 1 - |D| = n -1 - (q - 1) = n - q$.
		Thus $$\sum_{v \in Z_n} deg(v) = \sum_{v \in D} deg(v) + \sum_{v \in F} deg(v) + deg(0) = (k + 1)(q -1) + n - q + k -1 =$$
		$qk - k + q - 1 + n - q + k -1 = 2n - 2$. Let $E$ be the set of all edges of $m-G(H)$. Then it is known that $$|E| = \frac{\sum_{v \in Z_n} deg(v)}{2} = \frac{2n - 2}{2} = n-1.$$
		Since the $m-G(H)$ is connected with $n$ vertices and $n-1$ edges, we conclude that the $m-G(H)$ is a tree. It is known that the chromatic number of any tree is 2.
	\end{proof}
	
	In light of Theorem \ref{T2} and the proof of Theorem \ref{T3}, we have the following result.
\begin{cor}\label{C2}
	Let $H$ be a finite cyclic group of order $n \geq 2$, $m > 1$ be an integer such that every prime factor of $n$ is a prime factor of $m$, $k = gcd(m, n)$ (note that every prime factor of $n$ is a prime factor of $k$), and $n = qk$ for some integer $q \geq 1$. Then the following statements hold (note that the $m$-$G(H)$ is connected by Theorem \ref{T1}).
	\begin{enumerate}
		\item If $k > 2$, then $deg(0) = k  - 1 \geq 2$.
		\item Let $a\in H$ If $k \nmid a$ (in $Z$), then $deg(a) = 1$.
		\item Let $a \in H \setminus \{ 0\}$. If $k \mid a$, then $deg(a) = k + 1$.
		\item There are exactly $q - 1$ distinct elements in $H$ of degree $k + 1$.
		\item There are exactly $n - q$ distinct nonzero elements in $H$ of degree 1.
		\item If $k = 2$, then there are exactly $(n - q) + 1$ distinct elements in $H$ of degree 1.
	\end{enumerate}
\end{cor}
\begin{thm}\label{Tciso}
Let $H$ be a finite cyclic group of order $n \geq 2$, i.e., $H \cong Z_n$, $m > 1$ be an integer, and $k = gcd(m, n)$.  Assume the $m$-$G(H)$ is connected, i.e., every prime factor of $n$ is a prime factor of $m$. Then the $m$-$G(H)$ is graph-isomorphic to the $k$-$G(H)$.
\end{thm}
\begin{proof}
If $n \mid m$, then $k = n$ and the $m$-$G(H)$ $\cong$ $k$-$G(H)$ $\cong$ $K_{1, n-1}$.  Hence, assume that $k \not = n$. Let $w$ be the least positive integer such that $n \mid k^w$. Hence $w$ is the least positive integer such that $n \mid m^w$ by Lemma \ref{L2}. Since $H \cong Z_n$, let $f$: $Z_n$  $\rightarrow$ $Z_n$ such that $f(0) = 0$. Let $a \in Z_n \setminus \{0\}$. If $k\nmid a$ (in $Z$), then define $f(a) = a$ and $f(k^ia) = m^ia \in Z_n$ for every $1\leq i \leq w$. Note that $k^ia = 0$ (in $Z_n$) if and only if $m^ia = 0$ in ($Z_n$) by Lemma\ref{L2}. Hence, if $k^ia \not = 0$ in $Z_n$ for some $1 \leq i < w$, then  $deg(k^ia) = deg(f(m^ia)) = deg(m^ia) = k+1$ by Theorem \ref{T2}.  To show that $f$ is a bijective function, it suffices to show that $f$ is surjective. Let $v \in Z_n$. If $v = 0$, then $f(0) = 0$. Assume that $v \not = 0$. If $k\nmid v$ in $Z$, then $f(v) = v$.  If  $k\mid v$ in $Z$, then there is a $c \in Z_n$ such that $k\nmid c$ in $Z$ and $v = k^ic$ in $Z$ for some $1 \leq i < w$. Since $gcd (m^i, n) = k^i$, there is a $b \in Z_n$ such that $m^ib = k^ic= v$ in $Z_n$. If $k\nmid b$ in $Z$, then $f(k^ib) = m^ib = v$ in $Z_n$. Suppose that $k \mid b$ in $Z$. Since $k^{i+1} \mid m^ib$ in $Z$ and $k^{i+1} \nmid k^ic$ in $Z$,  we conclude that $k^{i+1} \nmid n$ in $Z$. Let $b + \frac{n}{k^i} = y\in Z_n$. Then $k \nmid y$ in $Z$ and $m^iy = k^ic = v \in Z_n$. Hence $f(k^iy) = m^iy = v$ in $Z_n$. Thus $f$ is surjective. Assume $d-kd$ is an edge in $k$-$G(H)$ for some $d \in Z_n - \{0\}$. Suppose that $deg(d) = 1$ (i.e., $k\nmid d$ in $Z$). Then $f(d)-f(kd) = d-md$ is an edge in $m$-$G(H)$.  Note that $kd \not = d$ in $Z_n$ and $md \not = d$ in $Z_n$ by Lemma \ref{L1}. Suppose that $k \mid d$ in $Z$. Then $d = k^ih$ (in $Z$)for some $1 \leq i <w$  and positive integer $h$ such that $k \nmid h$ in $Z$. Thus $d-kd = k^ih-k^{i + 1}h$ is an edge in $k$-$G(H)$. Hence $f(k^ih)-f(k^{i + 1}h)) = m^ih-m^{i+1}h$ is an edge in $m$-$G(H)$. Thus the $m$-$G(H)$ is graph-isomorphic to the $k$-$G(H)$.

\end{proof}

\begin{thm}\label{D1}
Let $H$ be a finite cyclic group of order $n \geq 2$, i.e., $H \cong Z_n$, $m > 1$ be an integer, and $k = gcd(m, n)$.  Assume the $m$-$G(H)$ ($k$-$G(H)$) is connected, i.e., every prime factor of $n$ is a prime factor of $m$. Let $w$ be the least positive integer such that $k^w = 0$ in $Z_n$. Let $a \in Z_n \setminus \{0\}$ such that $k\nmid a$ in $Z$ and $r$ be the least positive integer such that $n \mid am^r$ ($ak^r$). Then $d(a, 0) = r \leq w$. In particular, $d(1, 0) = d(n-1, 0) = w$ and $diam$($m$-$G(H)$) = $diam$($k$-$G(H)$) $\leq 2w$.
\end{thm}
\begin{proof}
Since the $m$-$G(H)$ is graph-isomorphic to the $k$-$G(H)$ by Theorem \ref{Tciso}, it suffices to prove the claim for the $k$-$G(H)$. Let $w$ be the least positive integer such that $k^w = 0$ in $Z_n$. Let $a \in Z_n \setminus \{0\}$ such that $k\nmid a$ in $Z$ and $r$ be the least positive integer such that $n \mid am^r$ ($ak^r$). Since $a\not = 0$, $k(ak^i) = k^ia$ only when $i = r$ by Lemma \ref{L1}. Since the $k$-$G(H)$ is a tree by Theorem \ref{T3}, $a-ak-  \cdots - ak^r = 0$ is the unique path between $a$ and $0$ in $k$-$G(H)$. Hence $d(a, 0) = r \leq w$. Since $w$ is the least positive integer such that $n \mid k^w$, we conclude that $d(1, 0) = d(n-1,0)) = w$. Hence $diam$($m$-$G(H)$) = $diam$($k$-$G(H)$) $\leq 2w$.
		\end{proof}.
		\begin{defi}
			Let $n, p \geq 2$ be positive integers such that $p \mid n$. Suppose that $p^a \mid n$ for some positive integer $a\geq 1$, but $p^{a+1} \nmid n$, then we write $p^a\mid\mid n$.
		\end{defi}
		\begin{thm}\label{DIM0}
			Let $H$ be a finite cyclic group of order $n \geq 2$, $m > 1$ be an integer such that every prime factor of $n$ is a prime factor of $m$,  $k = gcd(m, n)$ (note that every prime factor of $n$ is a prime factor of $k$), and $n = qk$ for some integer $q \geq 1$. Then the following statements hold.
			\begin{enumerate}
			\item If $q = 1$ and $|H| = n = 2$, then $k = n = 2$ and the $m$-$G(H)$ $\cong$ $k$-$G(H)$ $\cong$  $K_{1, 1}$ and $diam$($m$-$G(H)) = diam$($k$-$G(H)) = 1$.
			\item If $q = 1$ and $|H| = n > 2$, then $k = n > 2$ and the $m$-$G(H)$ $\cong $ $k$-$G(H)$ $\cong K_{1, n-1}$ and $diam$($m$-$G(H)) = diam$($k$-$G(H)) = 2$.
			\item  If $k = 2 \not = n$, then $|H| = n = 2^w$ for some positive integer $w \geq 2$ and  $diam$($m$-$G(H)) = diam$($k$-$G(H)) = 2(w - 1)$. 			
			\item If $k > 2$ and $|H| = n = k^w$ for some positive integer $w \geq 1$, then  $diam$($m$-$G(H)) = diam$($k$-$G(H)) = 2w$.
			\item Let $w$ be the least positive integer such that $n \mid k^w$. Suppose there is an odd prime $p \mid n$ (and hence $p \mid k$) and a positive integer $a\geq 2$ such that $p^a\mid \mid n$, but $p^a \nmid k^{w-1}$.  Then $diam$($m$-$G(H)) = diam$($k$-$G(H)) = 2w$.
			\item Assume that $n$ is an odd integer, and $w$ is the least positive integer such that $n \mid k^w$. Then $diam$($m$-$G(H)) = diam$($k$-$G(H)) = 2w$.
			\item Suppose that $n$ is an even integer, $ n \not = k$, and $w$ is the least positive integer such that $n \mid k^w$.  Then one of the following holds.
			\begin{enumerate}
				\item Assume that $2k^{w-1} \not = 0$ in $Z_n$. Then $diam$($m$-$G(H)) = diam$($k$-$G(H)) = 2w$.
				\item Assume that $2k^{w-1} = 0$ in $Z_n$ and $k \not = 2$. Then $diam$($m$-$G(H)$) = $diam$($k$-$G(H)$) = $2w - 1$.
				\item Assume that  $2k^{w-1} = 0$ in $Z_n$ and $k = 2$. Then $diam$($m$-$G(H)$) = $diam$($k$-$G(H)$) = $2(w - 1)$.
			\end{enumerate}
			\end{enumerate}
		\end{thm}		
		\begin{proof} Since the $m$-$G(H)$ is graph-isomorphic to the $k$-$G(H)$ by Theorem \ref{Tciso}, it suffices to prove the claims for the $k$-$G(H)$.
			\newline
		(1) and (2). The proof is clear.
			\newline
			(3). Assume that $k = 2 \not = n$. Since every prime factor of $n$ is a prime factor of $k$, we conclude that $|H| = n = 2^w$ for some positive integer $w \geq 2$. Since $k = 2$, it is clear that $2^{w-1}$ is the only element in $Z_n$  adjacent to $0$ by Theorem \ref{T2}(1). Since $d(1, 0) = d(n - 1, 0) = w$ by Theorem \ref{D1} and $2^{w-1} = (n-1)2^{w-1}$ $= \frac{n}{2}$ in $Z_n$, we conclude that $1-2-2^2\cdots -2^{w-1} = (n-1)2^{w-1}-(n-1)2^{w-2}-\cdots-(n-1)2-(n-1)$ a path of maximum length in $k$-$G(H)$. Since $d(1, -1) = 2(w-1)$,   $diam$($m$-$G(H)) = diam$($k$-$G(H)) = 2(w - 1)$.
			\newline
			(4). Assume that  $k >  2$ and $|H| = n = k^w$ for some positive integer $w \geq 2$. Since $k > 2$, it is clear that $2-2k-2k^2-\cdots -2k^{w-1}-2k^w = 0$ is a path of length $w$ in $K$-$G(H)$. Since $d(1, 0) = w$ by Theorem \ref{D1}, we conclude that $1-k-k^2-\cdots - k^w=0-2k^{w-1}-2k^{w-2}-\cdots -2k-2$ is a path in $k$-$G(H)$ of length $2w$. Thus $diam$($m$-$G(H)) = diam$($k$-$G(H)) = 2w$ by Theorem \ref{D1}.
			\newline
			(5) Let $w$ is the least positive integer such that $n \mid k^w$. Assume there is an odd prime $p \mid n$ (and hence $p \mid k$) and a positive integer $a\geq 2$ such that $p^a\mid \mid n$, but $p^a \nmid k^{w-1}$ (note that $k^0 = 1$). Then it is clear that $1-k-k^2-\cdots - k^w=0=2k^w-2k^{w-1}-2k^{w-2}-\cdots -2k-2$ is a path in $k$-$G(H)$ of length $2w$. Hence $diam$($m$-$G(H)) = diam$($k$-$G(H)) = 2w$ by Theorem \ref{D1}.
			\newline
			(6) Assume that $n$ is an odd integer and $w$ is the least positive integer such that $n \mid k^w$. Then there  is an odd prime $p \mid n$ (and hence $p \mid k$) and a positive integer $a\geq 2$ such that $p^a \mid\mid n$, but $p^a \nmid k^{w-1}$. Hence $diam$($m$-$G(H)) = diam$($k$-$G(H)) = 2w$ by (5).
			\newline
			(7) Suppose that $n$ is an even integer, $n\not = 2$, and $w$ is the least positive integer such that $n \mid k^w$.
			\newline
				(a). Assume that $2k^{w-1} \not = 0$ in $Z_n$. Then it is clear that $1-k-k^2-\cdots - k^w=0=2k^w-2k^{w-1}-2k^{w-2}-\cdots -2k-2$ is a path in $k$-$G(H)$ of length $2w$. Hence $diam$($m$-$G(H)) = diam$($k$-$G(H)) = 2w$ by Theorem \ref{D1}.
				\newline
				(b). Since $k \not = n$, $w\geq 2$. Assume that $2k^{w-1} = 0$ in $Z_n$ and $k \not = 2$. Since $2k^{w-1}$, $vk^{w-1} = 0$ in $Z_n$ for every even integer $v \in Z_n$ and $dk^{w-1} = k^{w-1} = k^{w-1}$ in $Z_n$ for every odd integer $d\in Z_n$. Since $k$-$G(H)$ is a tree, every path of length $w$ from an element $a\in Z_n\setminus \{0\}$ to $0$ must pass through $k^{w-1}$. Hence $diam$($k$-$G(H)$) $ <  2w$. Since $k \not = 2$, $k^{w-1}$ and $2k^{w-2}$ (note that $2k^0 = 2 \not = k$ ) are two distinct element in $Z_n$ that are adjacent to $0$. Since $k(2k^{w-1}) = 2k^{w-1} = 0$ in $Z_n$, $d(2, 0) = w-1$. Hence $1-k-k^2-\cdots - k^w=2k^{w-1} = 0-2k^{w-2}-\cdots -2k-2$ is a path of maximum length in $k$-$G(H)$. Since $d(1, 2) = w + w-1 = 2w - 1$, we conclude that $diam$($m$-$G(H)) = diam$($k$-$G(H)) = 2w -1$.
				\newline
				(c). Since $k \not = n$, $w\geq 2$. Assume that $2k^{w-1} = 0$ in $Z_n$ and $k = 2$. Since every prime factor of $n$ is a prime factor of $k$, we conclude $|H| = 2^w$. Thus $diam$($m$-$G(H)) = diam$($k$-$G(H)) = 2(w - 1)$ by (3).
		\end{proof}
In light of Theorem \ref{DIM0}(3, 6, 7), we have the following
result.
\begin{cor}\label{C3}
	Let $H$ be a finite cyclic group of order $n \geq 2$, $m > 1$ be an integer such that every prime factor of $n$ is a prime factor of $m$,  $k = gcd(m, n)\not = n$ (note that every prime factor of $n$ is a prime factor of $k$), and $n = qk$ for some integer $q \geq 2$. Let $w$ be the least positive integer such that $k^w = 0$ in $Z_n$. The following statements hold.		

	\begin{enumerate}
		\item $diam$($m$-$G(H)$) = $diam$($k$-$G(H)$) $\in \{2w, 2w - 1, 2(w-1)\}$.
		
		\item $diam$($m$-$G(H)$) = $diam$($k$-$G(H)$) = $2w$ if and only if $n$ is an odd integer or $n$ is an even integer such that $2k^{w-1} \not = 0$ in $Z_n$.
		\item $diam$($m$-$G(H)$) = $diam$($k$-$G(H)$) = $2w - 1$ if and only if $n$ is an even integer, $2k^{w-1} = 0$ in $Z_n$, and $k \not = 2$.
		\item $diam$($m$-$G(H)$) = $diam$($k$-$G(H)$) = $2(w - 1)$ if and only if $n$ is an even integer, and $k = 2$ (and hence $|H| = 2^w$).
	\end{enumerate}
	\end{cor}
\begin{thm}\label{CDIM}
	Let $n > 2$, and $k\mid n$ such that every prime factor of $n$ is also a prime factor of $k$ (i.e., the $k$-$G()$ is connected), and $H = Z_n$. Let $i$ be the least positive integer such that $k^i\mid n$, and hence $n = qk^i$ for some integer $1 \leq q < k$. Let $w$ be the least positive integer such that $n\mid k^w$.  The following statements hold.
	 \begin{enumerate}
	 	\item If $q = 1$, then $w = i$. If $q\not = 1$, then $w = i+1$.
	 	\item If $q = 1$ and $k >2$, then $diam$($k$-$G(H)$) = $2i = 2w$.
	 	\item If $q = 1$ and $k = 2$, then $diam$($k$-$G(H)$) = $2(i - 1) = 2(w-1)$.
	 	\item If $q = 2$, then $diam$($k$-$G(H)$) = $2w - 1 = 2i + 1$.
	 	\item if $q > 2$, then $diam$($k$-$G(H)$) = $2w = 2(i+1)$.
	 \end{enumerate}
\end{thm}
\begin{proof}
	(1). Since $n = qk^i$ and $1\leq q < k$, it is clear that $w = i$ if $q = 1$, and $w = i+1$ if $q\not = 1$.
(2). Assume that $q = 1$ and $k >2$. Then $n = k^i = k^w$, Hence  $diam$($k$-$G(H)$) = $2i = 2w$ by Theorem \ref{DIM0}(4).
\newline
(3). Assume that $q = 1$ and $k = 2$. Since $|H| > 2$, $|H| = 2^i = 2^w$ and  $i = w \geq 2$. Thus  $diam$($k$-$G(H)$) = $2(i - 1) = 2(w - 1)$ by Theorem \ref{DIM0}(3).
\newline
(4). Assume that $q = 2$. Hence $w = i + 1$. Since $1 \leq q <k$, we conclude that $k \not = 2$. Since $2k^{w-1} = 2k^i = 0$ in $Z_n$ and $k \not = 2$, we conclude that $diam$($k$-$G(H)$) = $2w - 1$ =  $2i + 1$ by Corollary \ref{C3}(3).
\newline
(5). Assume that $q > 2$. Then $w = i + 1$. Since $2k^{w-1} = 2k^i \not = 0$ in $Z_n$, we conclude that $diam$($k$-$G(H)$) = $2w$ =  $2(i + 1)$ by Corollary \ref{C3}(2).
	\end{proof}
	\begin{thm}\label{DIM2}
		Let $H$ be a finite cyclic group of order $n > 2$, i.e., $H \cong Z_n$, $m > 1$ be an integer, and $k = gcd(m, n)$.  Assume the $m$-$G(H)$ is connected, i.e., every prime factor of $n$ is a prime factor of $m$. Then $diam$($m$-$G(H)$) ($diam$($k$-$G(H))) = 2$ if and only either $k = 2$ and $n = 4$ or $k = n$ and $n \geq 3$.
	
	\end{thm}
	\begin{proof}
		Since $gcd(n, m) = k$, $n = qk^i$ for some positive integers $1 \leq q < k$ and $i \geq 1$. Assume that $k = 2$ and $n = 4$. Then $|H| = 2^2$, and hence $diam$($2$-$G(H)$) = $2 (2-1) = 2$ by Theorem \ref{DIM0}(3). Assume that $k = n$ and $n \geq 3$. Then the $n$-$G(H)$ $\cong$ $K_{1, n-1}$, and thus  $diam$($m$-$G(H)$) = $2$. Conversely, suppose that $diam$($m$-$G(H)$) = $2$. Then by Theorem \ref{CDIM}, we conclude that $q = 1$, $n = k^i$, and $1 \leq i \leq 2$. Assume that $i = 1$. Since $n > 2$, we have $n = k \geq 3$.  Assume that $i = 2$. Then by Theorem \ref{CDIM}((1) and (2)), we conclude that $k = 2$ and $|H| = 4$.  
		\end{proof}
\begin{exa}\label{E4}
Let $H = Z_{20}$, and $k = 10$. Then the $10$-$G(H)$ is connected by Theorem \ref{T1}. Since $20 = 2k$, we conclude that $w = 2$ is the least positive integer such that $20 \mid k^2$. Since $k > 2$, we conclude the diameter of $10-G(H)$ is $2w - 1 = 3$ by Theorem \ref{CDIM}(4). For example, $1-10-0-2$ is a path of length $3$.  
\end{exa}

In the following result, we show that for every integer $d \geq 1$, there is an $m > 1$ and a cyclic group $H$ of order $n$ such that the diameter of the $m$-$G(H)$ is $d$. 

\begin{thm}\label{T4} Let $d \geq 1$. Then there is an $m > 1$ and a cyclic group $H$ of order $n$ such that the $m$-$G(H)$ is connected and has diameter $d$. 
		\end{thm}
\begin{proof} Assume $d$ is an odd integer. If $d = 1$, then the $2$-$G(Z_2)$ is connected and has  diameter $d = 1$. Assume that $d \not = 1$. Since $d$ is odd, $d = 2i + 1$ for some integer $i \geq 1$. Let $n = 2(6^{i})$, $H = Z_n$, and $m = k = 6$. Then the $m$-$G(H)$ is connected by Theorem \ref{T1} and has diameter $2(i + 1) - 1 = 2i + 1 = d$ by Theorem \ref{CDIM}(4). Assume $d \geq 2$ is an even integer. Then $d = 2i$ for some integer $i \geq 1$. Let $n = 6^{i}$, $H = Z_n$, and $m = k = 6$. Then the $m$-$G(H)$ is connected by Theorem \ref{T1} and has diameter $2i = d$ by Theorem \ref{CDIM}(2).

\end{proof}

In the following results, we demonstrate that certain trees can be represented as the $m$-$G(H)$ of a cyclic group $H$ for some $m \geq 2$. 

\begin{thm}\label{Tree1}
Let $T$ be a tree with a root vertex $v_0$ that has a degree $d$ for some integer $d \geq 1$. Suppose that $v_0$ is not adjacent to every vertex of degree $1$, and $T$ has exactly $d$ vertices, each of which has a degree of $d + 2$.  All other vertices in $T$ are of degree $1$. Then $T$ can be represented as the $d+1$-$G(Z_n)$, where $n = (d+1)^2$. Furthermore, $diam(T)$ = $2$ if $d = 1$, and $diam(T)$ = $4$ if $d\not = 1$. 
\end{thm}
\begin{proof}
	Let $V$ be the set of all vertices of $T$, and let $F = \{v \in V | deg(v) = d+2\}$. Then $|F| = d$. Since $deg(v_0) = d$ and $v_0$ is not adjacent to every vertex of degree $1$, we conclude that $v_0$ is adjacent to every vertex in $F$. Hence, no two vertices in $F$ are adjacent, and every vertex $v \in V \setminus \{v_0\}$  of degree $1$ is adjacent to exactly one vertex in $F$. Thus, $|V| = d(d+1) + |F| + 1 = (d+1) ^2$. Since the $d+1$-$G(Z_{|V|})$ shares the same graphical properties as $T$ by Theorem \ref{T2}, we can express  $T$ as the $d+1$-$G(Z_n)$, where $n = |V| = (d + 1)^2$. By Theorem \ref{CDIM}((2) and (3)), it follows that $diam(T)$ = $2$ if $d = 1$, and $diam(T)$ = $4$ if $d\not = 1$. 
	
\end{proof}
The following result demonstrates that not every tree can be represented as the $m$-$G(H)$ of a cyclic group $H$ for some $m \geq 2$.
\begin{thm}\label{NoTree}
	Let $T$ be a tree with a root vertex $v_0$ that has a degree of $2^d - 1$ for some odd integer $d \geq 1$, and let $V$ be the set of all vertices of $T$. If $|V| = (2^d)b$ for some odd integer $b\geq 3$. Then $T$ cannot be represented as the $m$-$G(H)$ of a cyclic group $H$ for some $m \geq 2$.
\end{thm}
\begin{proof}
	Assume that $T$ can be represented as the $m$-$G(H)$ of a cyclic group $H$ for some $m \geq 2$. Then $H| = |V| = b(2^d)$, every prime factor $|H|$ is a prime of $k = gcd(m, |H|)$ by Theorem \ref{T1}, and $deg(0) = k - 1 = 2^d - 1$ by Theorem \ref{T2}. Thus $k = 2^n$. Since every odd prime factor of $n$ is not a factor of $k$, the $m$-$G(H)$ is disconnected by Theorem \ref{T1}. Thus, $T$ cannot be represented as the $m$-$G(H)$ of a cyclic group $H$ for some $m \geq 2$. 

\end{proof}

We have the following result.
\begin{thm}\label{Tree2}
	Let $T$ be a tree with a root vertex $v_0$ that has a degree of $k-1 \geq 3$ for some even integer $k\geq 4$, and $T$ has exactly $2k - 1$ vertices, each of which has a degree of $k + 1$.  Let $V$ be the set of all vertices of $T$, and $F = \{v \in V \mid deg(v) = k+1\}$ . Suppose that one vertex in $F$ is adjacent to $k$ vertices in $F$, and every vertex adjacent to $v_0$ has a degree of $k+1$.  All other vertices in $T$ are of degree $1$. Then $T$ can be represented as the $k$-$G(Z_n)$, where $n = 2k^2$. Furthermore, $diam(T)$ = $5$.
\end{thm}
\begin{proof}
	 Let $a$ be the vertex in $F$ that is adjacent to $k$ vertices in $F$. Hence, $deg(v_0) + deg(a) = |F| = 2k - 1$. Let $v$ be a vertex of $T$ that has a degree of $1$. Then, either $v$ is adjacent to a vertex $b \in F\setminus \{a\}$ such that $b$ is adjacent to $v_0$ or $v$ is adjacent to a vertex $c \in F\setminus \{a\}$ such that $c$ is adjacent to $a$. Hence, the number of all vertices of $T$ that have a degree of $1$ is $(k-2 + k)(k) = (2k - 2)(k) = 2k^2 - 2k$. Hence, $|V| = 2k^2 - 2 + |F| + 1 = 2k^2 - 2k + 2k -1 + 1 = 2k^2$. Since the $k$-$G(Z_{|V|})$ shares the same graphical properties as $T$ by Theorem \ref{T2}, we can express  $T$ as the $k$-$G(Z_n)$, where $n = |V| =2k^2$. By Theorem \ref{CDIM}(4), it follows that $diam(T)$ = $2(2) + 1 = 5$.
	
\end{proof}
\begin{remk}\label{rng}
Let $n = p_1^{n_1}\cdots p_m^{n_m}$, where the $p_i$'s are distinct prime integers, and each $n_i \geq 1$ for every $1\leq i \leq m$. Let $S = \{k \mid n : p_i \mid k$, for every $1\leq i \leq m\}$. Then, it is clear that $|S| = n_1\cdots n_m$. Thus, there are exactly $n_1\cdots n_m$ distinct connected $k$-$G(Z_n)$ by Theorem \ref{T1}.
	\end{remk}
\begin{exa}
	Let $n = 72 = 3^22^3$. Hence by Remark\ref{rng}, we have 6 distinct connected $k$-$G(Z_{72})$. All possible values of $k$ are $72, 36, 24, 18, 12, 6$. Let $k = 72$. Since the $72$-$G(Z_{72})$ = $K_{1, 71}$, the $diam$($72$-$G(Z_{72})$ = $2$. Let $k = 36$. Then $n = 72 = 2k$. Thus, the $diam$($36$-$G(Z_{72})$ = $3$ by Theorem \ref {CDIM}(4). Let $k = 24$. Then $n = 72 = 3k$. Thus, the $diam$($24$-$G(Z_{72})$ = $4$ by Theorem \ref {CDIM}(5). Let $k = 18 $. Since $n = 72 = 4k$, we have the $diam$($18$-$G(Z_{72})$ = $4$ by Theorem \ref {CDIM}(5). Let $k = 12$. Then $n = 72 = 6k$. Thus, the $diam$($12$-$G(Z_{72})$ = $4$ by Theorem \ref {CDIM}(5). Let $k = 6$. Since $n = 72 = 2k^2$, we have the $diam$($18$-$G(Z_{72})$ = $5$ by Theorem \ref {CDIM}(4).

\end{exa}

\section{Properties of the $m$-$G(H)$ when $H$ is not cyclic.}

We start with the following result.

\begin{thm}\label{T5}
	Let $n_1, ..., n_i$  be positive integers such that $i \geq 2$,  $gcd(n_e, n_j) \not = 1$ for every $e \not = j$, $1 \leq e, j \leq i$, and $n = n_1n_2 \cdots n_i$. Let $m$ be a positive integer such that every prime factor of $n$ is a prime factor of $m$, and $k = gcd(n, m)$. For each $1 \leq j \leq i$, let $d_j = gcd(m, n_j)$ and $n_j = q_jkd_j$ for some positive integer $q_j$. Let $H$ be the abelian group $Z_{n_1}  \times \cdots \times Z_{n_i}$. Then the following statements hold.
	\begin{enumerate}
		\item  $H$ is not cyclic and the $m$-$G(H)$ is connected.
		\item $deg((0, ..., 0)) = d_1d_2\cdots d_i  - 1 \geq 2^i -1$.
		\item Let $(a_1, ..., a_i) \in H$. If $d_j \nmid a_j$ (in $Z$) for some $1 \leq j \leq i$, then $deg((a_1, ..., a_i)) = 1$.
		\item Let $(a_1, ..., a_i) \in H \setminus \{ (0, ..., 0)\}$. If $d_j \mid a_j$ for every $1 \leq j \leq i$, then $deg((a_1, ... , a_i) = d_1d_2 \cdots d_i + 1$.
		\item There are exactly $q_1q_2 \cdots q_i - 1$ distinct elements in $H$ of degree $d_1d_2 \cdots d_i + 1$.
		\item There are exactly $n - q_1q_2 \cdots q_i$ distinct elements in $H$ of degree 1.
		\item The $m$-$G(H)$ is a tree.
	\end{enumerate}
\end{thm}

\begin{proof}
	 First, for every $1 \leq j \leq i$, observe that $d_j \geq 2$ and every prime factor of $n_j$ is a prime factor of $d_j$.
	 \vskip0.1in
	{\bf (1)}. It is clear that $k \mid n$ and every prime factor of $n$ is a prime factor of $k$. Hence by Theorem \ref{T1}, the $m$-$G(H)$ is connected. Since the $gcd(n_e, n_j) \not = 1$ for every $e \not = j$, $1 \leq e, j \leq i$, we conclude that $H$ is not cyclic.
	\vskip0.1in
	{\bf (2)} The equation $mX = m(x_1,..., x_i) = (mx_1, ... , mx_i) = (0, ..., 0)$ has a solution in $H$ if and only if $(d_1x_1, ... , d_ix_i) = (0, ... , 0)$ has a solution in $H$ by Lemma \ref{L1}. For each $1 \leq j \leq i$, the equation $d_jx_j = 0$ has exactly $d_j \geq 2$ distinct solutions (including 0) in $Z_{n_j}$ by Lemma \ref{L1}. Thus the number of distinct nonzero elements that are adjacent to $(0, ... , 0)$ is $d_1d_2 \cdots d_i - 1$. Since $d_j \geq 2$ for every $1\leq j \leq i$ and $i \geq 2$, we have $deg((0, ... , 0)) = d_1d_2 \cdots d_i - 1 \geq 2^i - 1$.

{\bf (3)}. Let $(a_1, ..., a_i) \in H$ and assume that $d_j \nmid a_j$ (in $Z$) for some $1 \leq j \leq i$. Hence $(a_1, ... , a_i) \not = (0, ... , 0)$. Since $d_jx_j = a_j$ has no solution in $Z_{n_j}$ by Lemma \ref{L1}, we conclude that $mX = (mx_1, ..., mdx_i) = (a_1, ... , a_i)$ has no solution in $H$. Since $d_j \nmid a_j$, $a_j \not = 0$. Thus $m(a_1, ..., a_j, ..., a_i) \not = (a_1, ..., a_j, ... , a_i)$ by Lemma \ref{L1}. Hence $m(a_1, ..., a_j, ..., a_i)$ is the only element in $H$ that is adjacent to $(a_1, ..., a_j, ..., a_i)$. Thus $deg((a_1, ..., a_j, ..., a_i)) = 1$.
\vskip0.1in
{\bf (4)}. Let $(a_1, ..., a_i) \in H \setminus \{ (0, ..., 0)\}$ and assume that $d_j \mid a_j$ for every $1 \leq j \leq i$. The equation $mX = m(x_1,..., x_i) = (mx_1, ... , mx_i) = (a_1, ..., a_i)$ has a solution in $H$ if and only if $(d_1x_1, ... , d_ix_i) = (a_1, ... , a_i)$ has a solution in $H$ by Lemma \ref{L1}. For each $1 \leq j \leq i$, the equation $d_jx_j = a_j$ has exactly $d_j$ distinct solutions  in $Z_{n_j}$ by Lemma \ref{L1}. Thus there are $d_1d_2 \cdots d_i$ distinct elements in $H$ that are adjacent to $(a_1, ... , a_i)$. Since $m(a_1, ..., a_i) \not = (a_1, ..., a_i)$ by Lemma \ref{L1}, we conclude that $m(a_1, ..., a_i)$ is adjacent to $(a_1, ..., a_i)$. Thus $deg((a_1, ... , a_i) = d_1d_2 \cdots d_i + 1$.
\vskip 0.1in
{\bf (5)}. Let $L = \{(a_1, ..., a_i) \in H \mid \   d_j \mid a_j$ \ for every $1 \leq j \leq i\}$. Let $w \in L$, then $w = (c_1d_1, ..., c_id_i)$ such that for each $1 \leq j \leq i$, we have $1 \leq c_j \leq q_j$. Hence $|L| = q_1q_2 \cdots q_i$. Thus $|L \setminus (0, ... , 0)| = q_1q_2 \cdots q_i - 1$. Thus there are exactly $q_1q_2 \cdots q_i - 1$ distinct elements in $H$ of degree $d_1d_2 \cdots d_i + 1$.

\vskip 0.1in
{\bf (6)}.  Let $L = \{(a_1, ..., a_i) \in H \mid \   d_j \mid a_j$ \ for every $1 \leq j \leq i\}$ and $U = \{(a_1, ..., a_i) \in H \mid \   d_j \nmid a_j$ \ for some $1 \leq j \leq i\}$. It is clear that $L \cup U = H$ and $L \cap U = \{\}$. Since $|L| = q_1q_2 \cdots q_i$ by (5) and $|H| = n$, we conclude that $|U| = n - |L| = n - q_1q_2 \cdots q_i$.  Since each element in $U$ is of degree one by (3), we conclude that there are exactly $n - q_1q_2 \cdots q_i$ distinct elements in $H$ of degree 1.

\vskip 0.1in
{\bf (7)}. Let $E$ be the set of all edges of $m$-$G(H)$. We show $|E| = |H| - 1 = n - 1$. It is known that $|E| = \frac{\sum_{v \in H} deg(v)}{2}$. Let $L$ and $U$ as in (6). Then by (2) - (6), we have
$$\sum_{v \in H} deg(v) = \sum_{v \in L} deg(v) + \sum_{v \in U} deg(v) = (q_1q_2 \cdots q_i - 1)(d_1d_2 \cdots d_i + 1) + 1(d_1d_2\cdots d_i - 1) + $$\\ $$(n - q_1q_2 \cdots q_i)1 = q_1q_2 \cdots q_i d_1d_2\cdots d_i + q_1q_2 \cdots q_i - d_1d_2\cdots d_i - 1 + d_1d_2\cdots d_i - 1 +$$ \newline   $$n - q_1q_2\cdots q_i = 2n - 2$$.

Hence $|E| = \frac{2n - 2}{2} = |H| -1 = n-1$. Since the $m$-$G(H)$ is connected and $|E| = |H| - 1|$, we conclude that the $m$-$G(H)$ is a tree.

\end{proof}
\begin{defi}
	Let $m_1, ..., m_i$ be positive integers such that $i \geq 2$, $m_1 \geq 2$, and $m_1 \mid m_2 \mid \cdots \mid m_i$. Let $H = Z_{m_1} \oplus \cdots \oplus Z_{m_i}$. For each $1 \leq j \leq i$, let $d_j \geq 2$ be a factor of $n_j$ such that every prime factor of $n_j$ is a prime factor of $d_j$. Let $d = d_1 \cdots d_i$. We define the {\it $(d_1, ..., d_i)$-product graph} on $H$, denoted by $(d_1, ... , d_i)$-$PG(H)$, to be undirected simple graph with vertex set $H$ and two distinct vertices $x = (x_1, ..., x_i), y = (y_1, ... , y_i) \in H$ are adjacent if $y=(d_1x_1, ..., d_ix_i)$ or $x = (d_1y_1, ... , d_iy_i)$. Since each $d_j$-$G(Z_{m_j})$, $1\leq j \leq i$, is connected by Theorem \ref{T1}, it is easy to see that the $(d_1, ... , d_i)$-$PG(H)$ is connected.
		\end{defi}
		\begin{remk}\label{r1}
			Let $H$ be a non-cyclic abelian group of order $n > 1$ such that $n$ is not a prime integer. It is known that there are positive integers, $m_1, ..., m_i$, such that $i \geq 2$, $m_1 \geq 2$, $m_1 \mid m_2 \mid \cdots \mid m_i$, $n = m_1\cdots m_i$, and $H\cong Z_{m_1} \oplus \cdots \oplus Z_{m_i}$.
			\end{remk}
			\begin{thm}\label{TNCiso}
				Let $H$ be a non-cyclic abelian group of order $n >1$ such that $n$ is not a prime integer, i.e., $H \cong S = Z_{m_1} \oplus \cdots \oplus Z_{m_i}$ such that $i \geq 2$, $m_1 \geq 2$, $m_1 \mid m_2 \mid \cdots \mid m_i$, and $n = m_1\cdots m_i$ by Remark \ref{r1}. Let $m > 1$ be a positive integer such that every prime factor of $m$ is a prime factor of $n$, and $k = gcd(m, n)$. For each $1\leq j \leq i$, let $d_j = gcd(m, n_j) = gcd(k, n_j)$. Then the $m$-$G(H)$ $\cong$ $k$-$G(H)$ $\cong$ $(d_1, ... , d_i)$-$PG(H)$.
			\end{thm}
			\begin{proof}
				If $n \mid m$, then $k = n$ and the $m$-$G(S)$ $\cong$ $k$-$G(S)$ $\cong$ $(d_1, ... , d_i)$-$PG(S)$ $\cong$  $K_{1, n-1}$.  Hence, assume that $k \not = n$. We show  the $k$-$G(S)$ $\cong$ $(d_1, ... , d_i)$-$PG(S)$. For every $1 \leq j \leq i$, let $f_j : Z_{m_j} \rightarrow Z_{m_j}$ such that $f_j(0) = 0$. Let $a \in Z_{m_j} \setminus \{0\}$. If $d_j\nmid a$ (in $Z$), then define $f_j(a) = a$ and $f_j(k^wa) = d_j^wa \in Z_{m_j}$ for every $w \geq 1$. Then by the proof of Theorem \ref{Tciso}, each $f_j$, $1 \leq j \leq i$, is a bijective function. Let $f$: $S$  $\rightarrow$ $S$ and $a = (a_1, ... , a_i) \in S$. Define $f(a) = (f_1(a_1), ... , f_i(a_i))$. Since each $f_j$, $1 \leq j \leq i$ is bijective, $f$ is bijective.  Assume $v-w$ is an edge in the $k$-$G(S)$ for some $v = (v_1, ... , v_i), w = (w_1, ... , w_i) \in S$. Hence, for every $1\leq j \leq i$,  $f_j(v_j)-f_j(w_j)$ is an edge in $d_j$-$G(Z_{m_j})$ by Theorem \ref{Tciso}. Thus $f(v)-f(w)$ is an edge in $(d_1, ... , d_i)$-$PG(S)$. Hence, the $k$-$G(S)$ $\cong$ $(d_1, ... , d_i)$-$PG(S)$.  Using a similar argument, we conclude that the $m$-$G(S)$ $\cong$ $(d_1, ... , d_i)$-$PG(S)$. Thus, the $m$-$G(H)$ $\cong$ $k$-$G(H)$ $\cong$ $(d_1, ... , d_i)$-$PG(H)$.
			\end{proof}
			\begin{lem}\label{L3}
			Let $m_1, m_2, k$ be positive integers such that $m_1 \mid m_2$ and every prime factor of $m_2$ is a prime factor of $k$. Let $d_1 = gcd(m_1, k)$, $d_2 = gcd(m_2, k)$, $w_1$ be the smallest positive integer such that $m_1 \mid d_1^{w_1}$, and $w_2$ be the smallest positive integer such that $m_2 \mid d_2^{w_2}$. Then $d_1 \mid d_2$ and $w_1 \leq w_2$.
			\end{lem}
			\begin{proof}
			Since every prime factor of $m_2$ is a prime factor of $k$ and $m_1 \mid m_2$, every prime factor of $m_1$ is a prime factor of $k$ and $d_1 = gcd(m_1, k) = gcd(m_1, d_2)$. Thus $d_1 \mid d_2$. Let $p$ be a prime factor of $m_1$. Then $p$ is a prime factor of $d_1$, $d_2$, $m_2$, and there are positive integers $a, b \geq 1$ such that $p^a \mid\mid m_1$ and $p^b \mid\mid m_2$.  Hence $p^a \mid p^b \mid p^{w_2}$. Thus, $m_1 \mid d_1^{w_2}$.  Hence $w_1 \leq w_2$.
			\end{proof}
			\begin{thm}\label{NCDIM}
			Let $H$ be a non-cyclic abelian group of order $n >1$ such that $n$ is not a prime integer, i.e., $H \cong S = Z_{m_1} \oplus \cdots \oplus Z_{m_i}$ such that $i \geq 2$, $m_1 \geq 2$, $m_1 \mid m_2 \mid \cdots \mid m_i$, and $n = m_1\cdots m_i$ by Remark \ref{r1}. Let $m > 1$ be a positive integer such that every prime factor of $m$ is also a prime factor of $n$, and $k = gcd(m, n)$. For each $1\leq j \leq i$, let $d_j = gcd(m, n_j) = gcd(k, n_j)$, and $w_j$ be the least positive integer such that $m_j \mid d_j^{w_j}$. The following statements hold.

			\begin{enumerate}
				
				\item If $w_{i-1} = w_i$ or $diam$($d_i$-$G(Z_{m_i})$) = $2w_i$, then $diam$($m$-$G(H)$) = $diam$($k$-$G(H)$) = $dim$($(d_1, ... , d_i)$-$PG(H)$) = $w_{i-1} + w_i = 2w_i$.
				\item If $w_{i-1} = w_i - 1$ and $diam$($d_i$-$G(Z_{m_i})$) $\not = 2w_i$, then  $diam$($m$-$G(H)$) = $diam$($k$-$G(H)$) = $dim$($(d_1, ... , d_i)$-$PG(H)$) $= 2w_i - 1$.
				\item If $w_{i-1} = w_i - c$ for some integer $c \geq 2$, then  $diam$($m$-$G(H)$) = $diam$($k$-$G(H)$) = $dim$($(d_1, ... , d_i)$-$PG(H)$) $=$ $diam$($d_i$-$G(Z_{m_i})$).
			\end{enumerate}
			\end{thm}
			\begin{proof}
			First, observe that the $k$-$G(H)$ $\cong$ $(d_1, ... , d_i)$-$PG(H)$ is a tree. For each $1\leq j \leq i$, let $d_j(a, 0)$ be the distance between $a \in Z_{m_j} \setminus \{0\}$ and 0 in $d_j$-$G(Z_{m_j}$). Hence $d_j(1, 0) = w_j$ and $d_j(a, 0) \leq w_j$ for every $a\in Z_{m_j} \setminus \{0\}$ by Theorem \ref{D1}. Since $w_1 \leq \cdots \leq w_i$ for every $1 \leq j \leq i$ by Lemma \ref{L3}, we conclude that $d(v=(v_1, ... , v_i), (0, .., 0)) \leq w_i$.  Let $v, w \in S\setminus \{(0, ..., 0)\}$. Then $d(v, w) \leq d(v, (0, ..., 0)) + d(w, (0, ... , 0)) \leq w_i + w_i = 2w_i$. Thus $diam$($m$-$G(H)$) =$diam$($k$-$G(H)$) = $dim$($(d_1, ... , d_i)$-$PG(H)$) $\leq$  $2w_i$. We consider the following cases. {\bf Case one}: Assume that $w_{i-1} = w_i$ or $diam$($d_i$-$G(Z_{m_i})$) = $2w_i$. If $diam$($d_i$-$G(Z_{m_i})$) = $2w_i$, then it is clear that $diam$($m$-$G(H)$) = $diam$($k$-$G(H)$) = $dim$($(d_1, ... , d_i)$-$PG(H)$) =  $2w_i$. Assume that $w_{i-1} = w_i$. Let $v = (0, ... , 1, 0)$, $u = (0, ..., 0, 1)$, and $O = (0, ... , 0)$. Then $d_{i-1}(v, O) = w_{i-1} = w_i$. Since $w_1 \leq \cdots \leq w_{i-1} \leq w_i$ by Lemma \ref{L3}, we have  $v-vd_{i-1}-\cdots - vd_{i-1}^{w_i}=O=ud_i^{w_i}-ud_i^{w_i - 1}-\cdots -ud_i-u$ is a path of length $2w_i$ in $(d_1, ... , d_i)$-$PG(H)$. Thus $diam$($m$-$G(H)$) = $diam$($k$-$G(H)$) = $dim$($(d_1, ... , d_i)$-$PG(H)$) =  $2w_i$. {\bf Case two}: Assume that $w_{i-1} = w_i - 1$ and $diam$($d_i$-$G(Z_{m_i})$) $\not = 2w_i$. Let $v, u, O$ be as in case one. Since $w_1 \leq \cdots \leq w_{i-1} \leq w_i$ by Lemma \ref{L3}, we have $v-vd_{i-1}-\cdots - vd_{i-1}^{w_i-1}= O = ud_i^{w_i}-ud_i^{w_i - 1}-\cdots -ud_i-u$ is a path of maximum length in $(d_1, ... , d_i)$-$PG(H)$. Hence $diam$($m$-$G(H)$) = $diam$($k$-$G(H)$) = $dim$($(d_1, ... , d_i)$-$PG(H)$) =  $2w_i - 1$. {\bf Case three}. Assume that $w_{i-1} = w_i - c$ for some integer $c \geq 2$. Then $Z_{m_i}$ $\not = Z_2$. Thus $m_i > 2$. Hence $2(w_i - 1)$ $\leq$ $diam$($d_i$-$G(Z_{m_i})$ $\leq 2w_i$ by Theorem \ref{CDIM}. Let $v = (v_1, , v_{i-1}, 0)$ and $u = (u_1, ... , u_i)$ be distinct nonzero elements in $S$, and $O = (0, ..., 0)$. Since $w_{i-1} = w_i - c$, $d(v, u) \leq d(v, O) + d(O, u) \leq  w - c + w \leq w-2 + w = 2(w-1)$ $\leq$ $diam$($d_i$-$G(Z_{m_i})$. Thus, a path of maximum length in $(d_1, ... , d_i)$-$PG(H)$ is of length $diam$($d_i$-$G(Z_{m_i})$. Hence $diam$($m$-$G(H)$) = $diam$($k$-$G(H)$) = $dim$($(d_1, ... , d_i)$-$PG(H)$) =  $diam$($d_i$-$G(Z_{m_i})$).
			
			\end{proof}
			We have the following examples.
			\begin{exa}
			Let $H = Z_4\oplus Z_{8}\oplus Z_{72}$, and $k = 6$. Then the $6$-$G(H)$ is connected by Theorem \ref{T1}. Let $m_1 = 4$, $m_2 = 8$, $m_3 = 72$, and $d_j = gcd(m_j, 6)$, for each $1\leq j \leq 3$. Then $d_1 = 2$, $d_2 = 2$, and $d_3 = 6$. For every $1\leq j \leq 3$, let $w_j$ be the least positive integer such that $m_j \mid d_j^{w_j}$. Then $w_1 = 2$, $w_2 = 3$, and $w_3 = 3$. Since $72 = 2k^2$, we have $diam$($d_3$-$G(Z_{72})$) = $2w_3 - 1 = 5$ $\not = 2w_3$ by Theorem \ref{CDIM}(4). Since $w_2 = w_3 = 3$, we conclude that $diam$($k$-$G(H)$) = $dim$($(d_1, d_2, d_3)$-$PG(H)$) = $2w_3 = 6$ by Theorem \ref{NCDIM}(1). Note that $deg((0, 0, 0)) = d_1d_2d_3 - 1 = 23$ by Theorem \ref{T5}(2), and the $k$-$G(H)$ has exactly $95$ element in $H$ of degree $d_1d_2d_3 + 1 = 25$ by Theorem \ref{T5}(5).  
					
			\end{exa}
			\begin{exa}
				Let $H = Z_8\oplus Z_{16}$, and $k = 2$. Then the $2$-$G(H)$ is connected by Theorem \ref{T1}. Let $m_1 = 8$, $m_2 = 16$, and $d_j = gcd(m_j, 2)$, for each $1\leq j \leq 2$. Then $d_1 = 2$, and $d_2 = 2$. For every $1\leq j \leq 2$, let $w_j$ be the least positive integer such that $m_j \mid d_j^{w_j}$. Then $w_1 = 3$, and $w_2 = 4$. Thus,  $diam$($d_2$-$G(Z_{16})$) = $2(w_2 - 1) = 6$ $\not = 2w_2$ by Theorem \ref{CDIM}(3). Since $w_1 = w_2 - 1 = 3$, we conclude that $diam$($k$-$G(H)$) = $dim$($(d_1, d_2)$-$PG(H)$) = $2w_2 - 1 = 7$ by Theorem \ref{NCDIM}(2). Note that $deg((0, 0)) = d_1d_2 - 1 = 3$ by Theorem \ref{T5}(2), and the $k$-$G(H)$ has exactly $31$ element in $H$ of degree $d_1d_2 + 1 = 5$ by Theorem \ref{T5}(5).  
				
			\end{exa}
			\begin{exa}
				Let $H = Z_4\oplus Z_{128}$, and $k = 4$. Then the $4$-$G(H)$ is connected by Theorem \ref{T1}. Let $m_1 = 4$, $m_2 = 128$, and $d_j = gcd(m_j, 4)$, for each $1\leq j \leq 2$. Then $d_1 = 4$, and $d_2 = 4$. For every $1\leq j \leq 2$, let $w_j$ be the least positive integer such that $m_j \mid d_j^{w_j}$. Then $w_1 = 1$, and $w_2 = 4$. Since $128 = 2k^3$, we have $diam$($d_2$-$G(Z_{128})$) = $2w_2 - 1 = 7$ $\not = 2w_3$ by Theorem \ref{CDIM}(4). Since $w_1 = w_2 - 3 = 1$, we conclude that $diam$($k$-$G(H)$) = $dim$($(d_1, d_2)$-$PG(H)$) = $diam$($d_2$-$G(Z_{128})$) = $2w_2 - 1 = 7$ by Theorem \ref{NCDIM}(3). Note that $deg((0, 0)) = d_1d_2 - 1 = 15$ by Theorem \ref{T5}(2), and the $k$-$G(H)$ has exactly $31$ element in $H$ of degree $d_1d_2 + 1 = 17$ by Theorem \ref{T5}(5).  
				
			\end{exa}
\vskip0.2in
	{\bf Acknowledgment}: The author is supported by the American University of Sharjah Research Fund FRG-21 AS1617.


\begin{thebibliography}{99}
	
	
	\bibitem{A1} Abawajy, J., Kelarev, A., Chowdhury, M. (2013). Power graphs:	A survey. Ejgta. 1(2): 125-147.
	
	\bibitem{AS} Anderson, D. F, Al-Kaseasbeh, S. (2023). The intersection subgroup graph of a group. {\it Commun. Algebra}. 51(8): 3556-3573.
	\bibitem{B1} Anderson, D. F., Asir, T., A. Badawi, A., Tamizh Chelvam, T. (2021). Graphs from rings. Springer Nature, Cham, Switzerland, 1 edition.
		
\bibitem{C0} Kumar, A., Selvaganesh, L., Cameron, P. J., Tamizh Chelvam, T.(2021). Recent developments on the power graph of finite groups – a survey. {\it AKCE International Journal of Graphs and Combinatorics}. 18(2):65-94.
\bibitem{C1} Chakrabarty, I., Ghosh, S., Sen, M. K. (2009). Undirected power graphs of semigroups. {\it Semigroup Forum}. 78(3): 410-426.


\bibitem{K1} Kelarev, A. V, Quinn, S. J. (2000). A combinatorial property
and power graphs of groups. {\it Contrib. General Algebra}. 12(58):
3-6.
\bibitem{K2} Kelarev, A. V., Quinn, S. J. (2002). Directed graphs and combinatorial properties of semigroups. {\it J. Algebra}. 251(1): 16-26.

	\bibitem{P1} Panda, R. P., Krishna K. V. (2018). On connectedness of power graphs of finite groups. {\it J. Algebra Appl.} 17(10): 1850184.
	
\bibitem{P2} Panda, R. P., Krishna, K. V. (2018). On the minimum degree,
	edge-connectivity and connectivity of power graphs of finite
	groups. {\it Commun. Algebra}. 46(7): 3182-3197.
	\bibitem{P3} Panda, R. P., Patra, K. L., Sahoo, B. K. (2021). On the minimum degree of the power graph of a finite cyclic group. {\it J. Algebra Appl.} 20(03): 2150044.
	\bibitem{P4} Pourghobadi, K, Jafari, S. H. (2018). The diameter of proper power graphs of symmetric groups. {\it J. Algebra Appl.} 17(12): 1850234.
	
	\bibitem{T1} Tamizh Chelvam, T., Sattanathan, M. (2013). Power graph of finite abelian groups. {\it Algebra Discret. Math}. 16: 33-41.

	
\end{thebibliography}
\end{document}